\documentclass[reqno, 12pt]{amsart}

\usepackage{_pascal/_pascal}

\hypersetup{
    pdftitle={Tree-independence number and forbidden induced subgraphs: excluding a 6-vertex path and a (2,t)-biclique},
    pdfauthor={Maria Chudnovsky, Julien Codsi, J. Pascal Gollin, Martin Milanič, Varun Sivashankar}
}

\newcommand{\tw}[0]{\ensuremath{\mathsf{tw}}}

\newcommand{\inducedsubgraph}[0]{\ensuremath{\subseteq_{\mathsf{i}}}}

\newcounter{tbox}

\makeatletter
\newcommand{\leqnomode}{\tagsleft@true}
\newcommand{\reqnomode}{\tagsleft@false}
\makeatother

\begin{document}

% \linenumbers

\author[Chudnovsky]{Maria Chudnovsky$^\ast$} 
\address[$\ast$]{Department of Mathematics, Princeton University, Princeton, NJ, USA. Supported by NSF Grants DMS-2348219 and CCF-2505100,  AFOSR grant FA9550-22-1-0083, and a Guggenheim Fellowship.
}
\email{\tt mchudnov@math.princeton.edu}

\author[Codsi]{Julien Codsi$^{\mathsection}$}
\email{\tt jc3530@princeton.edu}
\address[$\mathsection$]{
Department of Mathematics, Princeton University, Princeton, NJ, USA. Supported by NSF Grants DMS-2348219 and  CCF-2505100,
AFOSR grant FA9550-22-1-0083, and Fonds de recherche du Qu\'{e}bec Grant 321124.}
\author[Gollin]{J.~Pascal Gollin$^\dagger$}
\address[$\dagger$]{FAMNIT, University of Primorska, Koper, Slovenia. %
Supported in part by the Slovenian Research and Innovation Agency (I0-0035, research programs P1-0285 and P1-0383, and research projects J1-3003, J1-4008, J1-4084, J1-60012, and N1-0370) and by the research program CogniCom (0013103) at the University of Primorska.}
\email{\tt pascal.gollin@famnit.upr.si}

\author[Milani\v{c}]{Martin Milani\v{c}$^{\dagger,\ddagger}$}
\address[$\ddagger$]{IAM, University of Primorska, Koper, Slovenia}
\email{\tt martin.milanic@upr.si}

\author[Sivashankar]{Varun Sivashankar$^\parallel$}
\address[$\parallel$]{Department of Mathematics, Princeton University, Princeton, NJ, USA}
\email{\tt varunsiva@princeton.edu}

\title[Tree-independence number of $\{P_6,K_{2,t}\}$-free graphs]{Tree-independence number\\ and forbidden induced subgraphs:\\ excluding a $6$-vertex path and a~$(2,t)$-biclique}

% \date{\today}
\date{April 2, 2026}



\begin{abstract}
We show that for every positive integer~${t \geq 2}$ there exists an integer~$s$ such that every graph that contains no induced subgraph isomorphic to either the $6$-vertex path or the $(2,t)$-biclique, the complete bipartite graph~$K_{2,t}$, has tree-independence number at most~$s$. 
This result makes partial progress on a conjecture of Dallard, Krnc, Kwon, Milani\v{c}, Munaro, \v{S}torgel, and Wiederrecht. 
\end{abstract}

\maketitle

\section{Introduction}
\label{sec:intro}

Treewidth and related graph width parameters are of central importance for designing polynomial-time algorithms on restricted inputs for various \textsf{NP}-hard problems.
The celebrated Courcelle's theorem~\cite{Courcelle:1990:MSO1}, for instance, yields linear-time algorithm for any decision problem expressible in \textsf{MSO}$_2$ logic when the input graph is restricted to graphs of bounded treewidth, and similar results are known for optimization problems~\cite{ArnborgLS:1991:EasyTW}. 
In view of the fact that classes of bounded treewidth are necessarily sparse, several more general width parameters were introduced in the literature that can also be bounded on dense graph classes and for which algorithmic metatheorems capturing more restricted sets of problems than Courcelle's theorem hold.
This includes, among others, clique-width~\cite{CourcelleMR:2000}, mim-width~\cite{BergougnouxDJ:2023}, tree-independence number~\cite{LimaMMORS:2024}, and induced matching treewidth~\cite{BodlaenderFK:2026}.

The latter two of these parameters are defined similarly as treewidth, using the concept of tree decompositions of graphs, but with a different measure of the quality of the tree decomposition.
In the case of tree-independence number, instead of minimizing the maximum number of vertices in a bag of a tree decomposition, the goal is to minimize the maximum number of pairwise non-adjacent vertices in a bag.
The more general parameter induced matching treewidth is defined similarly, but with respect to induced matchings consisting only of edges intersecting a fixed bag. 

Both parameters admit algorithmic metatheorems.
As shown by Lima et al.~\cite{LimaMMORS:2024}, in any vertex-weighted graph with bounded tree-independence number, one can find in polynomial time a maximum weight induced subgraph with bounded treewidth and satisfying some \textsf{CMSO}$_2$ property.
For induced matching treewidth, a similar result was recently obtained by Bodlaender, Fomin, and Korhonen~\cite{BodlaenderFK:2026}.

Another way treewidth can be generalized is by requiring its boundedness only for subclasses of the considered class.
Graphs with large clique number necessarily have large treewidth.
A graph class is \emph{hereditary} if it is closed under induced subgraphs.
A hereditary graph class is said to be \emph{$(\tw,\omega)$-bounded} if the absence of large cliques is not only a necessary but also a sufficient condition for bounded treewidth.
Such graph classes have some good algorithmic properties for variants of the \textsc{Clique} and \textsc{Coloring} problems (see~\cite{ChaplickTVZ:2021,DallardMS:2024:TWvsOmega1}).
Furthermore, since $(\tw,\omega)$-boundedness is a necessary condition for bounded tree-independence number (see~\cite{DallardMS:2024:TWvsOmega2}), the aforementioned algorithmic metatheorem due to Lima et al.~remains valid if the condition on bounded treewidth is replaced with bounded clique number.

The full potential of the algorithmic usefulness of $(\tw,\omega)$-bounded classes is not yet understood.
This motivated Dallard et al.~\cite{DallardMS:2024:TWvsOmega3} to conjecture that for hereditary graph classes, $(\tw,\omega)$-boundedness is equivalent to bounded tree-independence number.
While the conjecture was recently disproved by Chudnovsky and Trotignon (see~\cite{ChudnovskyTrotignon:2025}), it is still open for classes defined by finitely many forbidden induced subgraphs, even in the following special case.
Given a family $\mathcal{F}$ of graphs, we say that a graph $G$ is \emph{$\mathcal{F}$-free} if no induced subgraph of~$G$ is isomorphic to a graph in $\mathcal{F}$.
We denote by $P_r$ the path with $r$ vertices and by $K_{s,t}$ the complete bipartite graph with parts of size $s$ and $t$, respectively.

\begin{conjecture}[Dallard et al.~\cite{DallardKKMMSW:2024:TWvsOmega4}]
    \label{conj:PrKtt}
    For any two positive integers~$r$ and~$t$, the class of~${\{P_r, K_{t,t}\}}$-free graphs has bounded tree-independence number.
\end{conjecture}

A weakening of this conjecture, establishing a polylogarithmic bound on the tree-independence number, follows from a result of Chudnovsky et al.~\cite[Theorem 1.2]{ChudnovskyCLMS:2025:TIN5}.
Dallard et al.~\cite{DallardKKMMSW:2024:TWvsOmega4} confirmed the validity of \zcref{conj:PrKtt} for ${r \le 4}$. 
Furthermore, they observed that for fixed~$r$ and~$t$, tree-independence number is bounded in any class of ${\{P_r,K_{1,t}\}}$\nobreakdash-free graphs (see also~\cite{BesterStorgelDLMZ:2025:awesome}) and suggested the following as an interesting open case.

\begin{conjecture}
    \label{conj:PrK2t}
    For any two positive integers~$r$ and~$t$, the class of ${\{P_r, K_{2,t}\}}$-free graphs has bounded tree-independence number.
\end{conjecture}

In this paper, we confirm the validity of this conjecture for~${r \le 6}$ and arbitrary~$t$. 
Prior to this work, the conjecture was open even for the case~${r = 5}$ and~${t \ge 3}$, as well as for the case~${r = 6}$ and~${t \ge 2}$; we refer to~\cite{HilaireMV:2025:TWvsOmega5} for more details and partial results. 

\begin{restatable}{theorem}{mainthm}
    \label{thm:main}
    For every integer~${t \geq 2}$ there exists an integer~$s$ such that every ${\{P_6,K_{2,t}\}}$\nobreakdash-free graph has tree-independence number at most~$s$. 
\end{restatable}

This result adds to the growing literature of results on bounding the tree-independence number in hereditary graph classes; see, for example, \cite{AbrishamiACHSV:TIN1,ChudnovskiHLS2026:TIN2,ChudnovskiHT2024:TIN3, ChudnovskyGHLS:2025:TIN4,ChudnovskyCLMS:2025:TIN5,ChudnovskyCLMS:2025:TIN6}.

Determining whether tree-independence number is bounded is an interesting question also when restricted to hereditary graph classes that are closed under \textsl{induced minors}.\footnote{A graph $H$ is an \emph{induced minor} of a graph $G$ if no graph isomorphic to $H$ can be obtained from an induced subgraph of $G$ by edge contractions (removing any resulting loops and fusing parallel edges).}               
In this regard, Dallard et al.~\cite{DallardMS:2024:TWvsOmega3} showed that for every positive integer $t$, excluding $K_{2,t}$ as an induced minor results in a class of bounded tree-independence number; and the same holds for the graph $K_5$ minus an edge and the graph obtained from a $4$-cycle by adding a universal vertex.
Furthermore, it follows from the main result of Chudnovsky et al.~\cite{ChudnovskyHS:2026:tw16} that excluding a path and a complete bipartite graph as an induced subgraph implies excluding the same path and a sufficiently large complete bipartite graph as an induced minor.
Hence, \zcref{conj:PrKtt} holds if (and only if) for every $r$ and~$t$, the class of graphs excluding $P_r$ and~$K_{t,t}$ as an induced minor has bounded tree-independence number.

Let us also mention a conjecture of Dallard et al.~\cite{DallardKKMMSW:2024:TWvsOmega4} generalizing \zcref{conj:PrKtt}, which states that excluding a star as an induced subgraph and a planar graph as an induced minor results in a graph class with bounded tree-independence number.
A weakening of this conjecture, establishing a polylogarithmic bound on the tree-independence number, was recently proved by Chudnovsky et al.~\cite{ChudnovskiCPR2026:TIN7}.

Our work leaves open the case $r\ge 7$ and $t \ge 2$ of \zcref{conj:PrKtt,conj:PrK2t}, as well as the case $r\ge 5$ and $t \ge 3$ of \zcref{conj:PrKtt}. 
In particular, the smallest open cases involve determining (un)boundedness of tree-independence number in the classes of $\{P_7,K_{2,2}\}$-free graphs and $\{P_5,K_{3,3}\}$-free graphs. 

\paragraph{Structure of the paper and overview of our techniques}
After introducing some necessary notation and tools in \zcref{sec:prelims}, we discuss the proof of the main result in \zcref{sec:main}. 
Our approach is based on a result from~\cite{ChudnovskiHLS2026:TIN2} (see \zcref{lem:small-tree-alpha}), which, roughly speaking, reduces the problem of bounding the tree-independence number in a certain class of graphs to showing the existence of so-called \textsl{balanced separators} with bounded independence number.
To this end, we prove in \zcref{lem:balanced-separator-small-alpha} a general sufficient condition for the existence of such separators.
This condition consists of two parts.
That each of the two parts is satisfied for any $\{P_6,K_{2,t}\}$-free graph is the statement of \zcref{lem:small-alpha-separator,lem:balanced-separator-closed-neighborhood}, respectively.
These two lemmas are then applied to prove \zcref{thm:main}. 
Finally, in \zcref{sec:pyramids} we prove \zcref{lem:small-alpha-separator,lem:balanced-separator-closed-neighborhood} by investigating the structure of so-called \textsl{pyramids} in~$\{P_6,K_{2,t}\}$-free graphs.

\section{Preliminaries}
\label{sec:prelims}

We take any standard notions and notation not explained here from~\cite{Diestel:GT5}. 

\paragraph{Basic notions and notation} 
Given an integer~$k$, we write~${[k] \coloneqq \{ i \in \mathbb{Z} \colon 1 \leq i \leq k\}}$ for the set of all positive integers less than or equal to~$k$. 
Given real numbers~$a$ and~$b$, we write~${[a,b) \coloneqq \{ x \in \mathbb{R} \colon a \leq x < b \}}$ for the half-open interval between~$a$ and~$b$. 

All graphs in this paper are finite and simple. 
Given a graph~$G$, we denote its vertex set by~$V(G)$ and its edge set by~$E(G)$. 
If a graph~$H$ is an induced subgraph of a graph~$G$, we write~${H \inducedsubgraph G}$. 
For convenience, slightly abusing the notation, we sometimes identify vertex subsets and the corresponding induced subgraphs. 
For a graph~$G$ and a set~${A \subseteq V(G)}$, we write~$G[A]$ for the induced subgraph of~$G$ with vertex set~$A$. 
Recall that for a family~$\mathcal{F}$ of graphs we say that a graph~$G$ is \emph{$\mathcal{F}$-free} if no induced subgraph of~$G$ is isomorphic to a graph in~$\mathcal{F}$. 
If~$\mathcal{F} = \{ H \}$ is a singleton, we also say that~$G$ is \emph{$H$-free}. 

For a vertex~${v \in V(G)}$, we denote by~$N_G(v)$ the \emph{(open) neighbourhood} of~$v$, that is, the set of vertices in~$G$ adjacent to~$v$, and by~$N_G[v]$ the \emph{closed neighbourhood} of~$v$, that is, the set~${N_G(v) \cup \{v\}}$. 
Similarly, for a set~${A \subseteq V(G)}$, 
the \emph{(open) neighbourhood} of~$A$ is the set~${N_G(A) \coloneqq \left(\bigcup_{v \in A} N_G(v) \right) \setminus A}$, 
and the \emph{closed neighbourhood} of~$A$ is the set~${N_G[A] \coloneqq \bigcup_{v \in A} N_G[v]}$. 
For all these notions, we drop the subscript if the ambient graph is clear from context. 
We say a vertex~${v \in V(G)}$ is \emph{simplicial} (in~$G$) if its neighbourhood is a clique. 

Given a graph~$G$ and a set~${A \subseteq V(G)}$, we denote the graph~${G[V(G) \setminus A]}$ obtained by deleting all vertices in~$A$ by~${G-A}$. 
When~${A = \{a\}}$ is a singleton, we may instead also write~${G-a}$. 

For a graph~$G$, we denote by~$\alpha(G)$ its \emph{independence number}, that is, the maximum size of an independent set in~$G$. 

\paragraph{Paths and separators} 
We denote a path as a string of its vertices, that is, if we have a set~$\{v_1, \dots, v_k\}$ of~$k$ distinct vertices, then ${v_1 v_2 \dots v_k}$ is the path with vertex set~$\{v_1, \dots, v_k\}$ and edge set~$\{ v_i v_{i+1} \colon i \in [k-1]\}$, where~$k$ is a positive integer. 
The \emph{length} of a path is the number of its edges. 
If~$P$ and~$Q$ are paths that share exactly one of their endpoints, we write the string~${P Q}$ to denote the \emph{concatenation} of~$P$ and~$Q$, that is, the unique path on~$V(P) \cup V(Q)$ that has both~$P$ and~$Q$ as induced subgraphs. 

An \emph{$(a,b)$-path} is a path whose set of endpoints is equal to~$\{a,b\}$. 
Given a graph~$G$ and sets~${A,B \subseteq V(G)}$ of vertices, an \emph{$(A,B)$-path} is an $(a,b)${\dash}path in~$G$ with~${a \in A}$, ${b \in B}$, and no internal vertex in~${A \cup B}$. 

A set~${S \subseteq V(G)}$ \emph{separates} $A$ from~$B$ if it intersects every $(A,B)$-path. 
When~${A = \{v\}}$ for a single vertex~$v$, we also just say that~$S$ separates~$v$ from~$B$. 

We say~$S$ is a \emph{minimal separator} if there are vertices~${u,v \in V(G)}$ such that~$S$ is an inclusion-wise minimal $(u,v)$-separator. 
It is well known that a set $S$ is a minimal separator if and only if there exist two distinct components~$C, C'$ of~${G-S}$ such that~${N_G(C) = N_G(C') = S}$. 
We call such components \emph{$S$-full}. 

\paragraph{The connectifier lemma}
A \emph{star} is a graph isomorphic to~$K_{1,\ell}$ for some positive integer~$\ell$. 
A \emph{comb} is a tree $C$ with maximum degree~$3$ and no vertices of degree $2$ such that there exists a path $P$ in $C$ such that all vertices of degree $3$ in $C$ belong to $P$.
The \emph{line graph} of a graph~$G$ is a graph~$L$ with vertex set~$E(G)$ and edge set~${\{ \{e,f\} \colon e\neq f, e \cap f \neq \emptyset \}}$, that is, $L$ is the \textsl{intersection graph} of~$E(G)$. 
A \emph{subdivision} of a graph~$H$ is a graph~$G$ obtained from~$H$ by replacing each edge~$uv$ of~$H$ with a $(u,v)$-path of length at least~$1$ in a way such that all these replacement paths are internally vertex-disjoint. 
If each replacement path has length exactly~$2$, we say that~$G$ is a \emph{$1$-subdivision of~$H$}. 
If~$G$ is a subdivision of~$H$, we also say~$G$ is a \emph{subdivided~$H$}.

\begin{lemma}[The connectifier lemma, {Chudnovsky et al.~\cite[Theorem 5.2]{ChudnovskyGHLS:2024:tw15}}]
    \label{lem:conntw15}
    For every integer~${h \geq 1}$, there exists an integer ${\mu(h) \geq 1}$ with the following property. 
    Let~$G$ be a connected graph and let~${Y \subseteq G}$ such that~${\abs{Y} \geq \mu(h)}$, ${G - Y}$ is connected, and every vertex of~$Y$ has a neighbor in~${G - Y}$. 
    Then there is a set~${Y' \subseteq Y}$ with~${\abs{Y'} = h}$ and an induced subgraph~$H$ of~${G - Y}$ such that one of the following holds. 
    \begin{itemize}
        \item $H$ is a path and every vertex of~$Y'$ has a neighbor in~$H$; or
        \item $H$ is a subdivided comb, the line graph of a subdivided comb, a subdivided star, or the line graph of a subdivided star. 
        Moreover, every vertex of~$Y'$ has a unique neighbor in~$H$ and~${H \cap N(Y')}$ is the set of simplicial vertices of~$H$. 
    \end{itemize}
\end{lemma}

\paragraph{Tree-decompositions, treewidth, and tree-independence number}
A \emph{tree-decomposition} of a graph~$G$ consists of a tree~$T$ and a non-null subtree~$T_v$ of~$T$ for each vertex~$v$ of~$G$ such that for every edge~$uv$ of~$G$, $T_u$ and $T_v$ have a common node.
For each node~$t$ of the tree~$T$, we let~${X_t \coloneqq \{v \colon t \in V(T_v)\}}$ be the \emph{bag} corresponding to~$t$.

The \emph{width} of a tree-decomposition as the maximum of~${\abs{X_t}-1}$ over the nodes~$t$ of the tree. 
The \emph{treewidth} of~$G$, denoted by~$\tw(G)$, is the minimum of the widths of its tree-decompositions. 

The \emph{$\alpha$-width} of a tree-decomposition is the maximum of~$\alpha(X_t)$ over the nodes~$t$ of the tree. 
The \emph{tree-independence number} of~$G$ is the minimum of the $\alpha$-widths of its tree-decompositions. 

\paragraph{Vertex-weighted graphs and balanced separators}
We consider graphs where the vertices are weighted. 
A \emph{weighting}~${\mathsf{w} \colon V(G) \to \mathbb{R}_{\geq 0}}$ on~$G$ is a map from~$V(G)$ to the non-negative reals. 
For a set~${X \subseteq V(G)}$ of vertices of~$G$, we write~${\mathsf{w}(X) \coloneqq \sum_{v \in X} \mathsf{w}(v)}$. 
We say~$\mathsf{w}$ is \emph{nontrivial} if~${\mathsf{w}(G) > 0}$, and \emph{trivial}, otherwise. 
For convenience, we often want to look at \emph{normal} weightings, that is, weightings with~${\mathsf{w}(G) = 1}$. 

Given a graph~$G$, a weighting~$\mathsf{w}$ of~$G$, and a real number~${c \in [\nicefrac{1}{2},1)}$, we call a set~${S \subseteq V(G)}$ of vertices a \emph{$(\mathsf{w},c)$-balanced separator of~$G$} if~${\mathsf{w}(D) \leq c\cdot\mathsf{w}(G)}$ for every component~$D$ of~${G - S}$. 

Clearly, given any graph~$G$ with a nontrivial weighting~$\mathsf{w}$ of~$G$, we can define a normal weighting~$\mathsf{w}'$ by setting~${\mathsf{w}'(v) = \nicefrac{\mathsf{w}(v)}{\mathsf{w}(G)}}$ for all $v\in V(G)$. 
It is now easy to see that for every~${c \in [\nicefrac{1}{2},1)}$ a set~${S \subseteq V(G)}$ is a $(\mathsf{w},c)$-balanced separator if and only if it is a $(\mathsf{w}',c)$-balanced separator. 
We observe also that for every~${c \in [\nicefrac{1}{2},1)}$, every~$(\mathsf{w},c)$-balanced separator is a $(\mathsf{w},c')$-balanced separator for every~$c'$ with~${c \leq c' < 1}$. 

A standard argument (e.g., by adapting the proof of {\cite[Lemma 7.19]{CyganFKLMPPS:2015:ParameterizedAlgorithms}) shows the following.

\begin{observation}\label{obs:tree-alpha-bs}
    Let~$G$ be a graph, let~${c \in [\nicefrac{1}{2},1)}$, and let~$\mathsf{w}$ be a normal weighting on~$G$. 
    If~$G$ has tree-independence number~$k$, then~$G$ has a ${(\mathsf{w},c)}$-balanced separator with independence number at most~$k$. 
\end{observation}

We will also use the following lemma to bound the tree-independence number when we have balanced separators. 

\begin{lemma}[{Chudnovsky et al.~\cite[Lemma 7.1]{ChudnovskiHLS2026:TIN2}}]
    \label{lem:small-tree-alpha}
    Let~$G$ be a graph, let~${c \in [\nicefrac{1}{2},1)}$, and let~$d$ be a positive integer. 
    If for every normal weighting~$\mathsf{w}$ on~$G$, 
    there is a ${(\mathsf{w},c)}$-balanced separator~$X$ with~${\alpha(X) \leq d}$, 
    then the tree-independence number of~$G$ is at most~${\frac{3-c}{1-c}d}$. 
\end{lemma}

An important ingredient of our proof is the following sufficient condition for bounded tree-independence number, which is an immediate consequence of~\cite[Lemma 3.2 and Theorem 3.11]{DallardMS:2024:TWvsOmega3}.

\begin{theorem}
    For every integer~${q \geq 2}$ and every graph~$G$ such that every minimal separator of~$G$ induces a subgraph with independence number less than~$q$, the tree-independence number of~$G$ is at most~${2q-2}$. 
\end{theorem}

We make use of the above result by combining it with \zcref{obs:tree-alpha-bs} into the following corollary. 

\begin{corollary}
    \label{cor:minimalseparatorboundedtreealpha}
    Let~${q \geq 2}$ be an integer, let~$G$ be a graph, and let~$\mathsf{w}$ be a normal weighting on~$G$. 
    If~${\alpha(S) < q}$ for every minimal separator~$S$ in~$G$, 
    then there exists a $(\mathsf{w},\nicefrac{1}{2})$-balanced separator~$Z$ in~$G$ with~${\alpha(Z) \leq 2q-2}$. 
\end{corollary}

\section{The main theorem}
\label{sec:main}

The  following lemma was proved (in a different context) in joint work of the first author with Sepehr Hajebi and Sophie Spirkl; it became the basis for Theorem 6.1 of \cite{ChudnovskiHLS2026:TIN2}. 
We thank Hajebi and Spirkl  for allowing us to include it here. 
For us, it outlines the big picture of our proof strategy when used in conjunction with \zcref{lem:small-tree-alpha}. 
As the context in~\cite{ChudnovskiHLS2026:TIN2} is slightly different, we include the proof of this lemma for completeness. 

\begin{lemma}
    \label{lem:balanced-separator-small-alpha}
    For every positive integers~$a$ and~$b$ and every~${c \in [\nicefrac{1}{2},1)}$ the following holds. 
    Let~$G$ be a graph such that
    \begin{enumerate}
        [label=(\alph*)]
        \item \label{item:lem:balanced-separator-small-alpha-a} for every pair~${u,v}$ of distinct and non-adjacent vertices of~$G$ there exists a $(u,v)$-separator ${A_{u,v} \subseteq V(G) \setminus \{u,v\}}$ such that ${\alpha(A_{u,v}) \leq a}$; and
        \item \label{item:lem:balanced-separator-small-alpha-b} for each connected induced subgraph~$G'$ of~$G$ and every normal weighting $\mathsf{w}'$ on~$G'$ there exists a vertex~${v \in V(G')}$ and a set~${B \subseteq V(G')}$ such that~${\alpha(B) \leq b}$ and ${N_{G'}[v] \cup B}$ is a ${(\mathsf{w}',c)}$-balanced separator of~$G'$. 
    \end{enumerate}
    Then, for every normal weighting~$\mathsf{w}$ on~$G$ there exists a $(\mathsf{w},c)$-balanced separator~$S$ in~$G$ such that~${\alpha(S) \leq a + 2b + 4}$. 
\end{lemma}

\begin{proof}
    Fix a normal weighting $\mathsf{w}$ on $G$. 
    First, let us define~${s \coloneqq a + 2b + 4}$. 
    
    Let~$T$ be the set of vertices~$v$ of~$G$ such that there exists a set~${B_v \subseteq V(G)}$ such that~${\alpha(B_v) \leq b + 1}$ and~${N[v] \cup B_v}$ is a ${(\mathsf{w},c)}$-balanced separator of~$G$. 
    Note that~$T$ is nonempty by property~\ref{item:lem:balanced-separator-small-alpha-b}. 
    We consider two cases. 
    
    \medskip
    
    \noindent\textit{Case 1: $T$ is not a clique.}
    Let~$v_1$ and~$v_2$ be two nonadjacent vertices in~$T$. 
    For~${i \in [2]}$, since~${v_i \in T}$, there exists a set $B_{v_i} \subseteq V(G)$ such that~${\alpha(B_{v_i}) \leq b+1}$ and~${N[v_i] \cup B_{v_i}}$ is a $(\mathsf{w},\frac{1}{2})$\nobreakdash-balanced separator of~$G$. 
    By property \ref{item:lem:balanced-separator-small-alpha-a}, there exists a $(v_1,v_2)$-separator ${A_{v_1,v_2} \subseteq V(G)\setminus\{v_1,v_2\}}$ such that~${\alpha(A_{v_1,v_2}) \leq a}$. 
    Let ${S \coloneqq A_{v_1,v_2} \cup \{v_1,v_2\} \cup B_{v_1} \cup B_{v_2}}$. 
    Then~${\alpha(S) \leq a+2b+4 = s}$. 
    We claim that~$S$ is the desired ${(\mathsf{w},c)}$-balanced separator of~$G$. 
    Assume for a contradiction that there exists a component~$C$ of~${G - S}$ such that~${\mathsf{w}(C) > c}$. 
    
    We claim that~$N(v_i) \cap C \neq \emptyset$ for each $i \in [2]$. 
    Indeed, suppose for a contradiction that~${N(v_i) \cap C = \emptyset}$ for some~${i \in [2]}$.
    Then~$C$ is a connected induced subgraph of~${G - (N[v_i]\cup B_{v_i})}$.
    Since $N[v_i]\cup B_{v_i}$ is a $(\mathsf{w},c)$-balanced separator of~$G$, it follows that~${\mathsf{w}(C) \leq c}$, a contradiction.
    This proves the claim.  
    
    Now~${C \cup \{v_1,v_2\}}$ is a connected graph, and so there exists a path~$P$ from~$v_1$ to~$v_2$ with~${V(P) \subseteq V(C) \cup \{v_1,v_2\}}$. 
    This contradicts the fact that~$v_1$ and~$v_2$ belong to different components of~${G - A_{v_1,v_2}}$. 
    
    \medskip
    
    \noindent\textit{Case 2: $T$ is a clique.}
    Since~${\alpha(T) \leq 1 \leq s}$, we may assume that $T$ is not a ${(\mathsf{w},c)}$-balanced separator of~$G$.
    Let~${G'}$ be the component of~${G - T}$ for which~$\mathsf{w}(G')$ is maximized. 
    Then ${\mathsf{w}(G') > c \geq \frac{1}{2}}$, and clearly all other components of~${G - T}$ have~$\mathsf{w}$-weight less than~${\nicefrac{1}{2} \leq c}$. 
    Let~$\mathsf{w}'$ be the function defined on~$G'$ by setting~${\mathsf{w}'(x) \coloneqq \mathsf{w}(x)/\mathsf{w}(G')}$ for all~${x \in V(G')}$. 
    Then~$\mathsf{w}'$ is a normal weighting on~$G'$. 
    Note that every vertex~${x \in V(G')}$ satisfies~${\mathsf{w}(x) \leq \mathsf{w}'(x)}$. 
    
    We apply property \ref{item:lem:balanced-separator-small-alpha-b} to~$G'$ and~$\mathsf{w}'$, and, hence, obtain a vertex~${v' \in V(G')}$ and a set~${B' \subseteq V(G')}$ such that~${\alpha(B') \leq b}$ and ${N_{G'}[v'] \cup B'}$ is a ${(\mathsf{w}',c)}$-balanced separator of~$G'$. 
    Let~${B_{v'} \coloneqq B' \cup T}$. 
    Note that 
    \[
          G - (N_G[v'] \cup B_{v'}) 
        = G - (N_G[v'] \cup B' \cup T) 
        = (G - T) - (N_{G'}[v'] \cup B') 
        = G' - (N_{G'}[v']\cup B')\,.
    \]
    In particular, every component~$C$ of~${G - (N_{G}[v'] \cup B_{v'})}$ is also a component of \linebreak ${G' - (N_{G'}[v'] \cup B')}$ and, hence, satisfies~${\mathsf{w}(C) \leq \mathsf{w}'(C) \leq c}$. 
    It follows that 
    \[
        {\alpha(B_{v'}) \leq  \alpha(B') + \alpha(T) \leq b+1}
    \] 
    and~${N_{G}[{v'}] \cup B_{v'}}$ is a $(\mathsf{w},c)$-balanced separator of~$G$, 
    contradicting the fact that~${v' \not\in T}$. 
\end{proof}

To be able to use this lemma, we need to be able to satisfy the assumptions in the class of ${\{P_6,K_{2,t}\}}$-free graphs, where~${t \geq 2}$. 
This is achieved by the following two lemmas, the proofs of which are postponed to the next section. 

\begin{restatable}{lemma}{mainlema}
    \label{lem:small-alpha-separator}
    Let~${t \geq 2}$, let~$G$ be a ${\{P_6,K_{2,t}\}}$-free graph, and let~$a$ and~$b$ be two distinct and nonadjacent vertices in~$G$. 
    Then, there exists an $(a,b)$-separator~${S \subseteq V(G) \setminus \{a,b\}}$ such that~${\alpha(S) \leq 14(t-1)+3}$. 
\end{restatable}

\begin{restatable}{lemma}{mainlemb}
    \label{lem:balanced-separator-closed-neighborhood}
    There exists a function~${g \colon \mathbb{N} \to \mathbb{N}}$ such that the following holds.
    Let~${t \geq 2}$, let~$G$ be a connected ${\{P_6,K_{2,t}\}}$-free graph, and let~$\mathsf{w}$ be a normal weighting on~$G$. 
    Then, there exists a vertex~${z_0 \in V(G)}$ and a set~${Z \subseteq V(G)}$ such that~${\alpha(Z) \leq g(t)}$ and~${N[z_0] \cup Z}$ is a ${(\mathsf{w},\nicefrac{7}{8})}$-balanced separator of~$G$. 
\end{restatable}

Before proving these two lemmas in the next section, 
let us give a formal proof of our main theorem using them. 

\mainthm*

\begin{proof}
    Using \zcref{lem:small-alpha-separator,lem:balanced-separator-closed-neighborhood} as well as \zcref{lem:balanced-separator-small-alpha}, for every~${\{P_6,K_{2,t}\}}$\nobreakdash-free graph~$G'$ and every normal weighting~$\mathsf{w}'$ there exists a $(\mathsf{w}',\nicefrac{7}{8})$-balanced separator~$S$ in~$G'$ such that~${\alpha(X) \leq 14 (t-1) + 3 + 2 g(t) + 4}$ where~$g$ denotes the function from \zcref{lem:balanced-separator-closed-neighborhood}. 
    
    Now, we can apply \zcref{lem:small-tree-alpha} to conclude that every~${\{P_6,K_{2,t}\}}$\nobreakdash-free graph has tree-independence number at most~${17(14(t-1) + 2 g(t) + 7) = 238(t-1) + 34 g(t) + 119}$. 
\end{proof}

\section{\texorpdfstring{%
Pyramids in $\{P_6,K_{2,t}\}$-free graphs}%
{Pyramids in { P\_6 , K\_2,t }-free graphs}}
\label{sec:pyramids}

Before proving \zcref{lem:small-alpha-separator,lem:balanced-separator-closed-neighborhood}, let us talk about particular induced subgraphs called pyramids. 
More precisely, in \zcref{subsec:pyramid-structure} we show why pyramids are useful towards our goal, and in \zcref{subsec:pyramid-finding} we show how to find them.

For an integer~${t \geq 2}$, a \emph{$t$-pyramid} is a graph~$\Pi$ on~${2t+1}$ vertices~$\{a, x_1, \dots, x_t, y_1, \dots, y_t\}$ with edge set ${\{ a x_i \colon i \in [t] \} \cup \{ x_i y_i \colon i \in [t] \} \cup \{ y_i y_j \colon i,j \in [t],\ i \neq j \}}$. 
To indicate the names of the vertices, we call the tuple ${(a,x_1,\dots,x_t,y_1,\dots,y_t)}$ the \emph{presentation} of~$\Pi$. 
Given a presentation ${(a,x_1,\dots,x_t,y_1,\dots,y_t)}$ of a pyramid $\Pi$, the vertex $a$ is said to be the \emph{apex} of~$\Pi$, and $\{y_1, \dots, y_t\}$ is its \emph{base}.
A \emph{pyramid} is a $3$-pyramid. 

Let~$G$ be a graph and let~${\Pi \inducedsubgraph G}$ be a pyramid. 
We say that~$\Pi$ is \emph{simplicial in~$G$} if for every~${v \in V(G) \setminus V(\Pi)}$ the neighborhood of~$v$ in~$\Pi$ is a non-empty clique. 
Moreover, we call a vertex in~${V(G) \setminus V(\Pi)}$ \emph{basic} (with respect to $\Pi$, or \emph{$\Pi$-basic} for short) if its neighbourhood in~$\Pi$ is exactly the base of~$\Pi$. 

\subsection{\texorpdfstring{%
Structure of pyramids in $\{P_6,K_{2,t}\}$-free graphs}%
{Structure of pyramids in { P\_6 , K\_2,t }-free graphs}}\label{subsec:pyramid-structure}

First, we show that if we have a pyramid~$\Pi$ in a $\{P_6,K_{2,t}\}$-free graph, then deleting some set with small independence number yields a component  in which~$\Pi$ is simplicial and every non-basic vertex is a neighbour of the apex. 

\begin{lemma}
    \label{lem:simplicial-pyramid-0}
     Let~$G$ be a $P_6$-free graph and let~${\Pi \inducedsubgraph G}$ be a pyramid with presentation~${(a,x_1,x_2,x_3,y_1,y_2,y_3)}$.    
     Let~${X \coloneqq \{x_1,x_2,x_3\}}$ and~${Y \coloneqq \{y_1,y_2,y_3\}}$ and let  \begin{itemize}
        \item ${Z \coloneqq \{ z \in N_G(V(\Pi)) \colon N_G(z) \cap V(\Pi) \textnormal{ is not a clique}\}}$. 
     \end{itemize}
     Then~$\Pi$ is simplicial in the component~$G'$ of~${G-Z}$ that contains~$\Pi$, and with the sets
     \begin{itemize}
        \item ${A \coloneqq N_{G'}(a) \setminus V(\Pi)}$ and
        \item ${B \coloneqq \{ b \in V(G') \setminus V(\Pi) \colon N_{G'}(b) \cap V(\Pi) = Y \}}$ (that is, $B$ is the set of $\Pi$-basic vertices in $G'$),
    \end{itemize}
    we have that~${\{A,B,V(\Pi)\}}$ is a partition of~$V(G')$. 
     
     Moreover, if~$G$ is~$K_{2,t}$-free for some integer~${t \geq 2}$, then~${\alpha(Z) \leq 12(t-1)}$.
\end{lemma}

\begin{proof}
    Let~${C \coloneqq V(G-Z) \setminus (A \cup B \cup V(\Pi))}$. 
    By the definition of~$Z$, for every~${v \in C}$, we have that~${N_{G}(v) \cap V(\Pi)}$ is a (possibly  empty) a clique. 
    To prove that~$\Pi$ is simplicial in~$G'$, it suffices to show that~$C$ is disjoint from~$V(G')$, since then~$\{A,B\}$ is a partition of~${N_{G'}(V(\Pi))}$. 

    \begin{claim}
        \label{clm:anticompletePi}
        ${N_G(C) \cap V(\Pi)}$ is empty. 
    \end{claim}

    \begin{subproof}
        For contradiction, assume that~${vw \in E(G)}$ with~${v \in C}$ and~${w \in V(\Pi)}$. 
        Recall that~${N_{G}(v) \cap V(\Pi)}$ is a clique by the definition of~$Z$, and it is non-empty, since ${w \in N_G(v) \cap V(\Pi)}$.
        
        By definition of~$A$, we conclude that~${a \notin N_{G}(v)}$, so~${w \neq a}$. 
        
        For all permutations~$(i,j,k)$ of~$\{1,2,3\}$, if~${x_i \in N_{G}(v)}$, then $N_{G}(v) \cap V(\Pi)$ is a clique containing $x_i$, which implies that~${y_{j},y_{k},x_{j},x_{k} \notin N_{G}(v)}$ and, hence, ${v x_i a x_j y_j y_k}$ is an induced path of length~$5$ in $G$, a contradiction. 
        Therefore, ${N_G(v)}$ is disjoint from $X$; in particular,~${w \notin X}$. 
        
        Hence,~${w \in Y}$, say~${w = y_1}$. 
        Since~$Y$ is a maximal clique in~$\Pi$, by definition of~$B$ and the fact that $v\notin B$, we conclude that~${Y \nsubseteq N_{G}(v)}$, say~${y_2 \notin N_G(v)}$. 
        But then, ${v y_1 y_2 x_2 a x_3}$ is an induced path of length~$5$ in $G$, a contradiction. 
    \end{subproof}

    \begin{claim}
        \label{clm:anticompleteA}
        ${N_G(C) \cap A}$ is empty. 
    \end{claim}

    \begin{subproof}
        For a contradiction, assume that~${vw \in E(G)}$ with~${v \in C}$ and~${w \in A}$.
        Since $w$ belongs to $A$, it does not belong to $Z$, hence, the set ${N_{G}(w) \cap V(\Pi)}$ is a non-empty clique.
        In particular, we have~${X \nsubseteq N_{G'}(w)}$, say~${x_1 \notin N_{G'}(w)}$. 
       For the same reason, $w$ is not adjacent to any vertex in $Y$. 
       Then, with \zcref{clm:anticompletePi}, the path~${v w a x_1 y_1 y_2}$ is an induced path of length~$5$ in $G$, a contradiction. 
    \end{subproof}

    \begin{claim}
        \label{clm:anticompleteB}
        ${N_G(C) \cap B}$ is empty. 
    \end{claim}

    \begin{subproof}
        For contradiction, assume that~${vw \in E(G)}$ with~${v \in C}$ and~${w \in B}$. 
        Then, it holds that ${N_{G}(w) \cap V(\Pi) = Y}$.
        With \zcref{clm:anticompletePi,clm:anticompleteA}, the path~${v w y_1 x_1 a x_2}$ is an induced path of length~$5$ in $G$, a contradiction. 
    \end{subproof}

    By \zcref{clm:anticompletePi,clm:anticompleteA,clm:anticompleteB}, ${N_G(C) \subseteq Z}$. 
    Hence, ${C \cap V(G') = \emptyset}$ and indeed~$\Pi$ is simplicial in~$G'$. 
    Moreover, ${V(G') = A \cup B \cup V(\Pi)}$, and the sets~$A$, $B$ and~$V(\Pi)$ are pairwise disjoint, so~${\{A,B,V(\Pi)\}}$ is indeed a partition of~$V(G')$. 

    What is left to show is the `moreover' part of the statement. 
    So assume~$G$ is $K_{2,t}$-free for some integer~${t \geq 2}$. 
    \begin{claim}
        \label{clm:Zsmallalpha}
        ${\alpha(Z) \leq 12 (t-1)}$ . 
    \end{claim}

    \begin{subproof}
        Let~$\overline{\Pi}$ denote the complement of~$\Pi$. 
        Note that~$\overline{\Pi}$ has~${\binom{7}{2}-9 = 12}$ edges. 
        For each~${uv \in E(\overline{\Pi})}$, let~${Z_{uv} \coloneqq \{ z \in Z \colon \{u,v\} \subseteq N_G(z) \}}$. 
        For every independent set~${I \subseteq Z_{uv}}$, the graph~${G[I \cup \{u,v\}]}$ is isomorphic to~$K_{2,\abs{I}}$, so~${\alpha(Z_{uv}) \leq t-1}$, since $G$ is  $K_{2,t}$-free. 
        
        By the definition of~$Z$, we have~${\bigcup_{uv \in E(\overline{\Pi})} Z_{uv} = Z}$, and hence
        \[
            \alpha(Z) \leq \abs{E(\overline{\Pi})} (t-1) \leq 12(t-1). \qedhere
        \]
    \end{subproof}
    This concludes the proof of the lemma. 
\end{proof}

Next we show that if we have a simplicial pyramid~$\Pi$ in a $\{P_6,K_{2,t}\}$-free graph, then every basic vertex can be separated from the apex by a set with small independence number. 

\begin{lemma}
    \label{lem:simplicial-pyramid-2}
    Let~${t \geq 3}$, let~$G$ be a ${\{P_6,K_{2,t}\}}$-free graph, and let~$\Pi$ be a pyramid with apex~$a$ such that $\Pi$ is simplicial in~$G$. 
    Then for every $\Pi$-basic vertex~${b \in V(G)\setminus V(\Pi)}$ there exists an $(a,b)$-separator~$S$ in~$G$ with~$\alpha(S) \leq 2(t-1) + 3$. 
\end{lemma}

\begin{proof}
    Let~$Y$ be the base of~$\Pi$ and let~${X \coloneqq V(\Pi) \setminus Y \cup \{a\}}$. 
    By \zcref{lem:simplicial-pyramid-0}, with
    \begin{itemize}
        \item ${A \coloneqq N_G(a) \setminus V(\Pi)}$ and
        \item ${B \coloneqq \{ b \in V(G) \setminus V(\Pi) \colon N_G(b) \cap V(\Pi) = Y \}}$, 
    \end{itemize}
    we observe that~${\{A,B,V(\Pi)\}}$ is a partition of~$V(G)$. 
    We claim that for every~${b \in B}$, the set 
    \begin{itemize}
        \item $S \coloneqq V(\Pi - a) \cup (N_G(b) \cap A) \cup (N_G(b) \cap B \cap N_G(A \setminus N_G(b)))$
    \end{itemize}
    is an $(a,b)$-separator of~$G$ with~${\alpha(S) \leq 2(t-1)+3}$. 

    Let~$P$ be an induced $(a,b)$-path in~${G - V(\Pi-a)}$. 
    Since~$N_G(a) \setminus V(\Pi) = A$, we conclude that~$P$ contains exactly one vertex in~$A$. 
    Note that this vertex is non-adjacent to some~${x \in X}$ since~$\Pi$ is simplicial in~$G$. 
    Then since~${x a P b}$ is an induced path longer than~$P$, we observe that~$P$ has length at most~${3}$, so~$P$ contains at most two vertices from~${B}$, one of which has to be~$b$. 
    If~$V(P) \cap B = \{b\}$, then the neighbor of~$b$ on~$P$ is contained in~${N_G(b) \cap A}$. 
    If~$V(P) \cap B = \{b,b'\}$ for some~${b' \in B}\setminus{b}$, then~${b' \in N_G(b) \cap B \cap N_G(A \setminus N_G(b))}$. 
    So $S$ is indeed an $(a,b)$-separator of~$G$. 

    For every independent set~${I \subseteq N_G(b) \cap A}$ we have that~$G[I \cup \{a,b\}]$ is isomorphic to~$K_{2,\abs{I}}$, so ${\alpha(N_G(b) \cap A) \leq t-1}$. 
    Let~${J \subseteq N_G(b) \cap B \cap N_G(A \setminus N_G(b))}$ be an independent set, and assume for a contradiction that~$J$ has size~$t$. 
    Let~${u \in J}$ and let~${v \in N_G(u) \cap (A \setminus N_G(b))}$.
    Note that~$v$ is non-adjacent to at least one vertex  in~$J$ since otherwise~${G[J \cup \{b,v\}]}$ would be isomorphic to~$K_{2,t}$.  
    Let ${w \in J}$ be such that~${w \notin N_G(v)}$. 
    Moreover, $u$ is non adjacent to some~${x \in X}$ since~$\Pi$ is simplicial in~$G$.
    But then, ${w b u v a x}$ is an induced path of length~$5$ in $G$, a contradiction. 
    So~${\alpha(N_G(b) \cap B \cap N_G(A \setminus N_G(b))) \leq t-1}$. 
    In total, we have 
    \[
        \alpha(S) \leq \alpha(\Pi-a) + \alpha(N_G(b) \cap A) + \alpha(N_G(b) \cap B \cap N_G(A \setminus N_G(b))) \leq 3 + 2(t-1)\,,
    \]
as claimed.
\end{proof}

\subsection{\texorpdfstring{%
Existence of pyramids in $\{P_6,K_{2,t}\}$-free graphs}%
{Existence of pyramids in { P\_6 , K\_2,t }-free graphs}}\label{subsec:pyramid-finding}

Our next lemma shows that, given two nonadjacent vertices~$a$ and~$b$ in a $\{P_6,K_{2,t}\}$-free graph, if we do not already find an $(a,b)$-separator with small independence number (as desired for \zcref{lem:small-alpha-separator}) in the neighbourhood of~$a$, then we find a pyramid such that~$a$ is its apex and~$b$ is basic. 

\begin{lemma}
    \label{lem:pyramids-exist-2}
    Let~${t \geq 2}$ be an integer and let~$G$ be a ${\{P_6,K_{2,t}\}}$-free graph containing two nonadjacent vertices~$a$ and~$b$ such that
    \begin{itemize}
        \item $N_G(a)$ is an independent set of size at least~$3(t-1)+1$ and 
        \item every vertex in~${N_G(a)}$ lies on some induced $(a,b)$-path.
    \end{itemize}
    Then, $G$ contains a pyramid~$\Pi$ with apex~$a$ as an induced subgraph such that $b$ is $\Pi$-basic. 
\end{lemma}

\begin{proof}
    We define
    \begin{itemize}
        \item $\mathcal{P}$ to be the set of all induced $(a,b)$-paths in~$G$, 
        % \item ${X \coloneqq \{ x \in N_G(a) \cap V(P) \colon P \in \mathcal{P},\, \abs{E(P)} = 3 \}}$, and 
        \item $X$ to be the set of neighbors of~$a$ such that there exists a path~${P \in \mathcal{P}}$ of length~$3$ with~${a \in P}$, and
        \item $Y$ to be the set of all vertices~$y$ such that there exists a path~${P \in \mathcal{P}}$ and an~${x \in X}$ with~${y \in N_P(x) \setminus \{a\}}$. 
    \end{itemize}
    Since $G$ is $P_6$-free, each path in $\mathcal{P}$ has length at most $4$. 
    In fact, as we show next, each path in $\mathcal{P}$ has length at most $3$. 
    Suppose for a contradiction that there exists a path~${a u v w b \in \mathcal{P}}$ of length $4$.
    Then, adding a vertex in~${N_G(a) \setminus \{u\}}$ to this path yields a (not necessarily induced) path of length $5$.
    Since~$G$ is~$P_6$-free, each vertex in~${N_G(a) \setminus \{u\}}$ is adjacent to one of~$v$, $w$, or~$b$. 
    But each of those $3$ vertices has at most $t-1$ neighbours in~${N_G(a)}$, as otherwise~$G$ would contain a~$K_{2,t}$. 
    Hence, $\abs{N_G(a)} \leq 3(t-1)$, contradicting the assumptions of the lemma. 
    
    Note that if there were~$t$ vertices in~$N_G(a)$ contained in a path~${P \in \mathcal{P}}$ of length~$2$, then~$G$ would contain an induced subgraph isomorphic to~$K_{2,t}$, a contradiction. 
    So we can conclude that~$X$ has cardinality at least~${2(t-1)+1}$. 
    
    Clearly, every vertex in~$X$ has a neighbor in~$Y$. 
    Let~$Y'$ be an inclusion-minimal subset of~$Y$ such that each vertex in~$X$ has a neighbour in~$Y'$. 
    Every vertex in~$Y'$ has at least~$1$ and at most~$t-1$ neighbors in~$X$ (at least $t$ neighbors would result in an induced subgraph isomorphic to $K_{2,t}$, as before). 
    Hence, ${\abs{Y'} \geq 3}$. 
    Moreover, ${Y' \subseteq N_G(b)}$ and~$({\{a\} \cup X) \cap N_G(b) = \emptyset}$. 
    So if we find a pyramid~${\Pi \inducedsubgraph G}$ with presentation ${(a,x_1,x_2,x_3,y_1,y_2,y_3)}$ where~${x_i \in X}$ and~${y_i \in Y'}$ for all~${i \in [3]}$, then that pyramid is as desired. 

    Let~$\{y_1, y_2, y_3\}$ be an arbitrary subset of~$Y'$ of size~$3$. 
    By the minimality of~$Y'$, each $y_i$ for~${i \in [3]}$ has a neighbour~$x_i \in X \setminus N_G(Y' \setminus \{y_i\})$. 
    If~$\{y_1,y_2,y_3\}$ is a clique in~$G$, we are done.     
    So assume~$y_1$ and~$y_3$ are not adjacent. 
    Then, the path~${y_1 b y_3 x_3 a x_2}$ is an induced~$P_6$ in~$G$ (see \zcref{fig:pyramidfinding-2}), a contradiction. 
\end{proof}

\begin{figure}[htbp]
    \centering
    \includegraphics[scale=1]{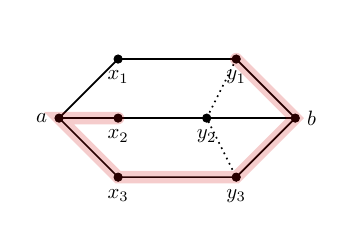}
    \caption{The situation in the proof of \zcref{lem:pyramids-exist-2}. 
    The dotted edges may or may not be edges of~$G$. }
    \label{fig:pyramidfinding-2}
\end{figure}

Next, we apply the connectifier lemma (\zcref{lem:conntw15}) to show that if in a $\{P_6,K_{2,t}\}$-free graph~$G$ we have a minimal separator with sufficiently large independence number, then there is a pyramid through the separator such that its apex and base are in different components. 

\begin{lemma}
    \label{lem:pyramids-exist-1}
    There exists a function~${f \colon \mathbb{N} \to \mathbb{N}}$ such that the following holds. 
    Let~${t \geq 3}$ be an integer, let~$G$ be a ${\{P_6,K_{2,t}\}}$-free graph, let~$S$ be a minimal separator of~$G$ with~${\alpha(S) \geq f(t)}$, and let~$C_1$ and~$C_2$ be two $S$-full components of~${G-S}$. 
    Then~$G$ contains a $t$-pyramid~$\Pi$ with presentation~$(a, x_1, \dots, x_t, y_1, \dots, y_t)$ as an induced subgraph such that
    \begin{itemize}
        \item ${V(\Pi) \cap S = \{x_1, \dots x_t\}}$, and 
        \item there exists an~${i \in [2]}$ for which~${a \in V(C_i)}$ and~${\{y_1, \dots, y_t\} \subseteq V(C_{3-i})}$. 
    \end{itemize}
\end{lemma}

\begin{proof}
    Let~$\mu$ be the function from \zcref{lem:conntw15}. 
    Without loss of generality, we may assume that~${\mu(n) \geq 5}$ for every~${n \in \mathbb{N}}$. 
    We define~${f(n) \coloneqq \mu(\mu(n^2)^2)}$ for every~${n \in \mathbb{N}}$. 

    Let~${Y \subseteq S}$ be an independent set of size~$f(t)$. 

    \setcounter{claim}{0}
    \begin{claim}\label{claim1}
        There is a subset~${X_1 \subseteq V(C_1)}$ and a subset~${Y_1 \subseteq Y}$ with~${\abs{Y_1} \geq \mu(t^2)}$ such that~${H_1 \coloneqq G[X_1 \cup Y_1]}$ is either a star, the $1$-subdivision of a star, or the line graph of the $1$-subdivision of a star, and the vertices in~$Y_1$ are precisely its vertices of degree~$1$ in~$H_1$. 
    \end{claim}

    \begin{subproof}
        If there is a vertex~${v_1 \in V(C_1)}$ with~${\abs{N(v_1) \cap Y} \geq \mu(t^2)}$, then we set~${X_1 \coloneqq \{v_1\}}$ and~${Y_1 \coloneqq N(v_1) \cap Y}$, and we observe that~${G[X_1 \cup Y_1]}$ is a star where~$Y_1$ is the set of leaves, as desired. 
        
        If there is no such vertex, we apply \zcref{lem:conntw15} to~$h = \mu(t^2)^2$ and~$G[V(C_1) \cup Y]$ to conclude that there is an induced subgraph~$H$ of~$C_1$ and a set~${Y' \subseteq Y}$ of size~${\mu(t^2)^2}$ such that either
        \begin{itemize}
            \item $H$ is a path and every vertex of~$Y'$ has a neighbor in~$H$; or
            \item $H$ is a subdivided comb, the line graph of a subdivided comb, a subdivided star, or the line graph of a subdivided star, every vertex of~$Y'$ has a unique neighbor in~$H$ and~${H \cap N(Y') = \mathcal{Z}(H)}$, where~$\mathcal{Z}(H)$ denotes the set of simplicial vertices of~$H$. 
        \end{itemize}
        
        We first show that $H$ is a subdivided star or the line graph of a subdivided star. 
        Suppose that this is not the case.
        Then, $H$ is a path, a subdivided comb, or the line graph of a subdivided comb. 
        Let~$A$ denote the set of vertices in~$H$ with neighbors in~$Y'$. 
        Then~${A\subseteq V(H)}$, with ${A = \mathcal{Z}(H)}$, unless~$H$ is a path. 
        Since~$G$ is $P_6$-free and $Y'$ is an independent set in $G$, if~$H$ is a subdivided comb or the line graph of a subdivided comb, then $H$ is $P_4$-free. 
        It is easy to observe that a $P_4$-free subdivided comb is a subdivision of a star, a contradiction.
        Furthermore, a $P_4$-free line graph of a subdivided comb has at most four simplicial vertices. 
        Hence, in all cases, ${\abs{A} \leq 5}$. 
        Since every vertex of~$Y'$ has a neighbor in~$H$ and each vertex of~$H$ has fewer than~$\mu(t^2)$ neighbors in~${Y \supseteq Y'}$, we have
        \[
            {\abs{Y'} \leq \abs{A} \cdot (\mu(t^2)-1) \leq 5 \cdot (\mu(t^2)-1) \leq \mu(t^2)(\mu(t^2)-1) < \mu(t^2)^2},
        \]
        a contradiction. 
        
        It follows that $H$ is a subdivided star or the line graph of a subdivided star. 
        Recall that each vertex in~$Y'$ has a unique neighbor in~$\mathcal{Z}(H)$, and each vertex in $\mathcal{Z}(H)$ has at least one neighbor in $Y'$. 
        Since each vertex of~$H$ has at most~$\mu(t^2)-1$ neighbors in~$Y'$, only simplicial vertices of~$H$ have neighbors in~$Y'$, and~${\abs{Y'} \geq \mu(t^2)^2}$, we conclude that~${\abs{\mathcal{Z}(H)} > \mu(t^2)}$. 
        For each~${v \in \mathcal{Z}(H)}$ we pick a vertex~${y_v \in Y' \cap N(v)}$.
    
        Suppose that~$H$ is a subdivided star. 
        Since $H$ is not a path, it has maximum degree at least $3$; hence, its center~$c$ is not simplicial and, consequently, no vertex in~$Y'$ is adjacent to~$c$. 
        Let~${X_1 \coloneqq V(H)}$ and~${Y_1 \coloneqq \{ y_v \colon v \in \mathcal{Z}(H) \}}$. 
        So~${H_1 \coloneqq G[X_1 \cup Y_1]}$ is a properly subdivided star. 
        Moreover, ${\abs{Y_1} = \abs{\mathcal{Z}(H)} \geq \mu(t^2)}$.
        If there is a vertex~${y \in Y_1}$ for which the path from~$y$ to~$c$ in~$H_1$ has length at least~$3$, then concatenating this path with any other path from~$c$ to~${y' \in Y_1\setminus\{y\}}$ yields an induced path of length at least~$5$ in $G$, a contradiction. 
        So~$H_1$ is the $1$-subdivision of a star, as desired. 

        Finally, suppose~$H$ is the line graph of a subdivided star. 
        Assume first that two distinct vertices~$v$ and~$w$ of~$H$ have degree~$1$. 
        Then there is an induced path in $G[V(H)\cup \{y_v,y_w\}]$ between~$y_v$ and~$y_w$ that has length at least~$5$, a contradiction. 
        So~$H$ has at most one vertex of degree~$1$. 
        Let~$v$ be the vertex of~$H$ of degree~$1$ if it exists, otherwise let~$v$ be an arbitrary vertex of~$H$. 
        Then,~${H-v}$ is a complete graph. 
        Let~${X_1 \coloneqq V(H-v)}$ and~${Y_1 \coloneqq \{ y_v \colon v \in X_1 \cap \mathcal{Z}(H) \}}$. 
        Note that ${\abs{Y_1} = \abs{\mathcal{Z}(H)}-1 \geq \mu(t^2)}$.
        But then, \linebreak ${H_1 \coloneqq G[X_1 \cup Y_1]}$ is the line graph of the $1$-subdivision of a star, as desired.   
    \end{subproof}

    With a completely analogous proof, we get the following. 
    
    \begin{claim}
        There is a subset~${X_2 \subseteq V(C_2)}$ and a subset~${Y_2 \subseteq Y_1}$ with~${\abs{Y_2} \geq t}$ such that ${H_2 \coloneqq G[X_2 \cup Y_2]}$ is either a star, the $1$-subdivision of a star, or the line graph of the $1$-subdivision of a star, and the vertices in~$Y_2$ are precisely its vertices of degree~$1$. 
    \end{claim}

    {\color{blue}Note that there exists a subset $X_1'$ of $X_1$ such that $G[X_1' \cup Y_2]$ is either a star, the $1$-subdivision of a star, or the line graph of the $1$-subdivision of a star (depending on whether $H_1$ is a star, the $1$-subdivision of a star, or the line graph of the $1$-subdivision of a star, respectively).} 
        Now consider~${H \coloneqq G[{\color{blue}X_1'} \cup Y_2 \cup X_2]}$. 
    If both~$H_1$ and~$H_2$ are stars, then~$H$ is isomorphic to~$K_{2,t'}$ for~${t' = \abs{Y_2} \geq t}$, a contradiction. 
    If one of~$H_1$ or~$H_2$ is a star and the other is the line graph of the $1$-subdivision of a star, then~$H$ is a $t'$-pyramid for~${t' = \abs{Y_2} \geq t \geq 3}$. 
    In all other cases, we easily find an induced path of length at least~$5$ in $G$, see \zcref{fig:pyramidfinding-1}. 
    This completes the proof of \zcref{lem:pyramids-exist-1}.
\end{proof}

\begin{figure}[htbp]
    \centering
    \includegraphics[scale=1]{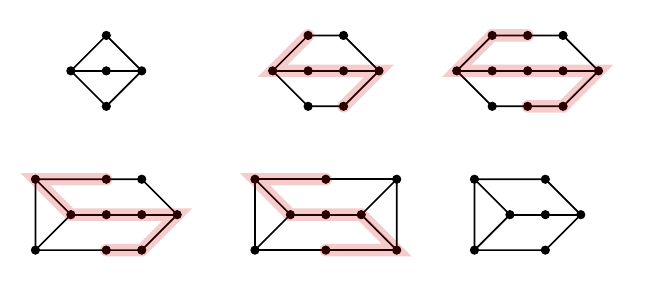}
    \caption{The six cases from the proof of \zcref{lem:pyramids-exist-1}, illustrated for the case $t = 3$. \\
    Top row: two stars ($K_{2,t}$), a star and a $1$-subdivision of a star (contains an induced $P_6$), and two $1$-subdivisions of a star (contains an induced~$P_9$). \\
    Bottom row: the line graph of the $1$-subdivision of a star and the $1$\dash{}subdivision of a star (contains an induced~$P_8$), two line graphs of $1$\dash{}subdivisions of stars (contains an induced~$P_7$), and a line graph of the $1$-subdivision of a star and a star (the desired pyramid).}
    \label{fig:pyramidfinding-1}
\end{figure}

\subsection{\texorpdfstring{%
Proof of \zcref{lem:small-alpha-separator}}{%
Proof of Lemma 3.1}}

We are now ready to prove \zcref{lem:small-alpha-separator}. 
To this end, using \zcref{lem:pyramids-exist-2}, between any pair~$a$, $b$ of nonadjacent vertices we find a pyramid with apex~$a$ with respect to which~$b$ is basic, for which then \zcref{lem:simplicial-pyramid-0,lem:simplicial-pyramid-2} allow us to find an $(a,b)$-separator with small independence number.

\mainlema*

\begin{proof}
    Let~$\mathcal{P}$ be the set of all induced $(a,b)$-paths in~$G$ and let~${X}$ be the set of all vertices~${x \in V(G)}$ such that there exists a path~${P \in \mathcal{P}}$ with~${x\in N_G(a) \cap V(P)}$. 
    Since~$a$ and~$b$ belong to different components of~${G - X}$, if~${\alpha(X) \leq  14 (t - 1) + 3}$, then~$X$ is the desired $(a,b)$-separator. 
    We can therefore assume that~${\alpha(X) \geq 14 (t - 1) + 3 \geq 3(t-1)+1}$. 
    Let~$I$ be an independent set in~$G$ such that~${\abs{I} = 3(t-1)+1}$ and~${I \subseteq {X}}$. 
    Let~$C$ be the component of~${G - {X}}$ containing~$b$. 
    
    Applying \zcref{lem:pyramids-exist-2} to~$G' \coloneqq G[\{a\} \cup I \cup V(C)]$, $a$, and~$b$ yields a pyramid~$\Pi$ with presentation~$(a,x_1,x_2,x_3,y_1,y_2,y_3)$ such that~$b$ is basic. 
    Clearly, $\Pi$ is an induced subgraph of~$G$. 

    By \zcref{lem:simplicial-pyramid-0}, there exists a set~${Z \subseteq V(G) \setminus (V(\Pi) \cup\{b\})}$ with~${\alpha(Z) \leq 12(t-1)}$ such that~$\Pi$ is simplicial in the component $D$ of~${G-Z}$ that contains~$\Pi$.
    Applying \zcref{lem:simplicial-pyramid-2} to $D$, there exists an $(a,b)$-separator~$S'$ in~$G-Z$ with~${\alpha(S') \leq 2(t-1) + 3}$. 
    Hence, we can conclude that~${S \coloneqq S' \cup Z}$ is an $(a,b)$-separator of~$G$ with ${\alpha(S) \leq 14(t-1) + 3}$, as desired. 
\end{proof}

\subsection{\texorpdfstring{%
Proof of \zcref{lem:balanced-separator-closed-neighborhood}}{%
Proof of Lemma 3.2}}

Lastly, we prove \zcref{lem:balanced-separator-closed-neighborhood}. 
Our strategy is to first apply \zcref{cor:minimalseparatorboundedtreealpha}, which allows us to assume that the graph has a minimal separator with large independence number.
This allows us to call upon \zcref{lem:pyramids-exist-1} to show that the graph contains a pyramid. 
Picking the right pyramid, we are then able to show that the set $Z$ of small  independence number guaranteed by 
\zcref{lem:simplicial-pyramid-0} together with a suitably chosen vertex $z_0$ leads to a balanced separator with the desired structure.

\mainlemb*

\begin{proof}
    Let~$f$ denote the function from \zcref{lem:pyramids-exist-1}. 
    We may assume that $f$ is nondecreasing and that~${f(t) \geq 12(t-1)}$ for all~${t \geq 3}$.
    We set $g(t) \coloneqq 2f(t)$ for every~${t \geq 3}$, and we set~${g(2) \coloneqq g(3)}$. 
    Since~$K_{2,2}$ is an induced subgraph of~$K_{2,3}$, every $K_{2,2}$-free graph is also $K_{2,3}$-free. 
    Hence, without loss of generality, we may assume that~${t \geq 3}$. 
    \setcounter{claim}{0}
    \begin{claim}
        \label{clm:WMA-lightclosedneighborhoods}
        We may assume that~${\mathsf{w}(N_G[v]) \leq \nicefrac{1}{8}}$ for all~${v \in V(G)}$. 
    \end{claim}

    \begin{subproof}
        Suppose there is a vertex~${v \in V(G)}$ such that~${\mathsf{w}(N_G[v]) > \nicefrac{1}{8}}$. 
        Then for every component~$C$ of~${G - N_G[v]}$ we have~$\mathsf{w}(C) \leq 1 - \mathsf{w}(N_G[v]) \leq \nicefrac{7}{8}$. 
        In particular, ${N_G[v] \cup \emptyset}$ is the desired $(\mathsf{w},\nicefrac{7}{8})$-balanced separator of~$G$. 
    \end{subproof}
    
    \begin{claim}
        \label{clm:WMA-minimalseparator}
         We may assume that~$G$ has a minimal separator~$S$ with~$\alpha(S) \geq f(t)$. 
    \end{claim}
    
    \begin{subproof}
        If for every minimal separator~$S$ in~$G$ we have~$\alpha(S) < f(t)$, then, by \zcref{cor:minimalseparatorboundedtreealpha}, we have that~$G$ has a $(\mathsf{w},\nicefrac{1}{2})$-balanced separator~$Z$ with~${\alpha(Z) \leq 2f(t)-2}\leq g(t)$. 
        Then, for any vertex $z_0\in V(G)$, the set $N[z_0]\cup Z$ is a $(\mathsf{w},\nicefrac{7}{8})$-balanced separator, as desired. 
    \end{subproof}

    By \zcref{clm:WMA-minimalseparator,lem:pyramids-exist-1}, $G$ contains a pyramid.     
    Let~${\Pi \inducedsubgraph G}$ be a pyramid with presentation~${(a,x_1,x_2,x_3,y_1,y_2,y_3)}$ such that~$\alpha(N_G(a))$ is maximized among all pyramids in~$G$. 
    Let~${Z \subseteq N_G(V(\Pi))}$ with~${\alpha(Z) \leq 12(t-1)}$ be such that~$\Pi$ is simplicial in the component~$G'$ of~${G - Z}$ that contains~$\Pi$ as in \zcref{lem:simplicial-pyramid-0}. 
    Note that for~${H \coloneqq G - (V(G')\cup Z)}$ we have~${N_G(H) \subseteq Z}$ since~$G'$ is a component of~${G - Z}$. 
    By \zcref{clm:WMA-lightclosedneighborhoods}, ${\mathsf{w}(N_G[V(\Pi)]) \leq \nicefrac{7}{8}}$. 
    If~$Z$ is a $(\mathsf{w},\nicefrac{7}{8})$-balanced separator of~$G$, then  for any vertex $z_0\in V(G)$, the set $N[z_0]\cup Z$ is a $(\mathsf{w},\nicefrac{7}{8})$-balanced separator of~$G$, and we are done, since $\alpha(Z) \leq g(t)$.
    So we may assume that~${G - Z}$ has a component~$C$ with~$\mathsf{w}(C) > \nicefrac{7}{8}$. 
    Note that~${C \neq G'}$, since $\Pi$ is simplicial in $G'$, so~$C$ is a component of~$H$. 
    Let~${z_0 \in Z \cap N_G(C)}$ be arbitrary (such a vertex exists since~$G$ is connected). 

    We claim that~${S \coloneqq N_G[z_0] \cup Z}$ a $(\mathsf{w},\nicefrac{7}{8})$-balanced separator. 

    \smallskip
    
    Suppose for a contradiction that it is not. 
    Let~$C'$ be the unique component of~${G - S}$ with~$\mathsf{w}(C') > \nicefrac{7}{8}$. 
    Since~${Z \subseteq S}$ and ${\mathsf{w}((G - S) - C) \leq \mathsf{w}(G - C) < \nicefrac{1}{8}}$ (as $\mathsf{w}(C) > \nicefrac{7}{8}$), we conclude that~${C' \subseteq C}$. 
    Let~${Z_1 \coloneqq N_G(z_0) \cap V(C) \cap N_G(C')}$. 
    Since~$C$ is connected and ${N_G(z_0) \cap V(C) \neq \emptyset}$, we infer that ${Z_1 \neq \emptyset}$. 
    Let~${z_1 \in Z_1}$ be arbitrary. 
            
    \begin{claim}
        \label{clm:WMA-Z0completetopyramid}
        ${V(G') \subseteq N_G(z_0)}$. 
    \end{claim}

    \begin{subproof}
        Assume for a contradiction that there exists~${v \in V(G') \setminus N_G(z_0)}$, and let~$P$ be an induced~$(v,z_0)$-path in the connected graph~${G[V(G') \cup \{z_0\}]}$. 
        By assumption, the length of~$P$ is at least~$2$. 
        
        Observe that~${C'\nsubseteq N_G[z_1]}$ as otherwise, ${\mathsf{w}(N_G[z_1]) \geq \mathsf{w}(C') > \nicefrac{7}{8}}$, contradicting \zcref{clm:WMA-lightclosedneighborhoods}. 
        Let~${w \in C' \setminus N_G[z_1]}$ and let~$Q$ be an induced $(w,z_1)$-path in~${G[C' \cup \{z_1\}]}$. 
        As before, the length of~$Q$ is at least~$2$. 
        Then the path~${v P z_0 z_1 Q w}$ is an induced path of length at least~$5$ in $G$, contradicting that~$G$ is $P_6$-free. 
    \end{subproof}

    \begin{claim}
        \label{clm:WMA-Z1largealpha}
        We may assume that~${\alpha(Z_1) > \alpha(Z)}$. 
    \end{claim}
    
    \begin{subproof}
        Suppose~${\alpha(Z_1) \leq  \alpha(Z)}$. 
        Since~${\alpha(Z) \leq 12 (t-1) \leq f(t)}$, we have 
        \[
           \alpha(Z \cup Z_1) \leq \alpha(Z) + \alpha(Z_1) \leq 2\alpha(Z) \leq 2f(t) = g(t)\,.
        \] 
        Consider~${S' \coloneqq N_G[z_1] \cup Z \cup Z_1}$. 
        If~$S'$ is a $(\mathsf{w},\nicefrac{7}{8})$-balanced separator of~$G$, then there is nothing else to show, since ${\alpha(Z \cup Z_1) \leq g(t)}$. 
        So we may assume that there exists a component~$C''$ of~${G - S'}$ with~${\mathsf{w}(C'') > \nicefrac{7}{8}}$. 
        Since ${N_G(C') \cap S \subseteq Z \cup Z_1 \subseteq S'}$ and \linebreak
        ${\mathsf{w}((G - S') - C') \leq \mathsf{w}(G - C') < \nicefrac{1}{8}}$ (as $\mathsf{w}(C') > \nicefrac{7}{8}$), we conclude that~${C'' \subseteq C'}$. 
        Let ${Z_2 \coloneqq N_G(z_1) \cap V(C') \cap N_G(C'')}$. 
        Clearly, ${Z_2 \neq \emptyset}$ since~$G$ is connected. 
        Let~${z_2 \in Z_2}$ be arbitrary. 
        Observe that~${C''\nsubseteq N_G[z_2]}$ as otherwise, hence~${\mathsf{w}(N_G[z_2]) \geq \mathsf{w}(C') > \nicefrac{7}{8}}$, contradicting \zcref{clm:WMA-lightclosedneighborhoods}. 
        Now, let ${w \in C'' \setminus N_G[z_2]}$ and let~$Q$ be an induced $(w,z_2)$-path in~${G[C'' \cup \{z_2\}]}$. 
        The length of~$Q$ is at least~$2$. 
        But then~${a z_0 z_1 z_2 Q w}$ is an induced path of length at least~$5$ in~$G$, contradicting that~$G$ is $P_6$-free. 
    \end{subproof}
    
    Now consider the graph~${G'' \coloneqq G[\{z_0\} \cup Z_1 \cup V(C')]}$. 
    By construction, $Z_1$ is a separator of~$G''$ and both components~$G[\{z_0\}]$ and~$C'$ of~${G'' - Z_1}$ are $Z_1$-full. 
    Hence, $Z_1$ is a minimal separator of~$G''$ and we can apply \zcref{lem:pyramids-exist-1} to find a $t$-pyramid, and hence a pyramid~$\Pi'$ with apex~$z_0$. 

    Now~${N_G(a) \setminus Z \subseteq V(G')}$, and by \zcref{clm:WMA-Z0completetopyramid}, 
    ${N_G(z_0) \supseteq V(G') \cup Z_1}$. 
    Thus, 
    \begin{align*}
        \alpha(N_G(z_0)) 
        \geq \alpha(N_G(a) \setminus Z) + \alpha(Z_1) \geq \alpha(N_G(a)) - \alpha(Z) + \alpha(Z_1) 
        > \alpha(N_G(a))\,,
    \end{align*}
    where the last inequality follows from \zcref{clm:WMA-Z1largealpha}.
    
    But this contradicts the choice of~$\Pi$, where we maximized the independence number of the neighborhood of the apex of the pyramid. 
    Hence, $S = N_G[z_0] \cup Z$ is indeed a $(\mathsf{w},\nicefrac{7}{8})$-balanced separator with~${\alpha(Z) \leq 12 (t-1)}$. 
\end{proof}

\printbibliography

\end{document}